\documentclass{amsart}
\usepackage{dsfont, amssymb, mathtools}
\usepackage{calc,changepage}
\usepackage{csquotes}
\usepackage{hyperref}
\usepackage{enumitem}
\setlist{  
  parsep=0pt,
}
\usepackage{xcolor}
\definecolor{wb}{RGB}{51,153,255}
\usepackage{tikz-cd}

\usepackage[parfill]{parskip}
\numberwithin{equation}{subsection}
\usepackage{amsrefs}

\newcommand{\defeq}{\vcentcolon=}

\makeatletter
\def\moverlay{\mathpalette\mov@rlay}
\def\mov@rlay#1#2{\leavevmode\vtop{%
   \baselineskip\z@skip \lineskiplimit-\maxdimen
   \ialign{\hfil$\m@th#1##$\hfil\cr#2\crcr}}}
\newcommand{\charfusion}[3][\mathord]{
    #1{\ifx#1\mathop\vphantom{#2}\fi
        \mathpalette\mov@rlay{#2\cr#3}
      }
    \ifx#1\mathop\expandafter\displaylimits\fi}
\makeatother

\newcommand{\bigcupdot}{\charfusion[\mathop]{\bigcup}{\cdot}}

\newtheoremstyle{definitions}
 	{\topsep}
	{\topsep}
	{}
	{}
	{\bfseries}
	{:}
	{.5em}
	{}
  
\newtheoremstyle{lemmata}
	{\topsep}
	{\topsep}
	{\itshape} 
	{}
	{\bfseries}
	{:}
	{.5em}
	{}

\theoremstyle{lemmata}
\newtheorem{Theorem}[subsection]{Theorem}

\newtheorem{Corollary}[subsection]{Corollary}
\newtheorem{Proposition}[subsection]{Proposition}

\theoremstyle{definitions}
\newtheorem{Definition}[subsection]{Definition}
\newtheorem{Remark}[subsubsection]{Remark}
\newtheorem{Remarks}[subsection]{Remarks}
\newtheorem*{Notation}{Notation}
\newtheorem{Example}[subsection]{Example}

\DeclareMathOperator{\tor}{tor}
\DeclareMathOperator{\spec}{spec}
\DeclareMathOperator{\rk}{rk}
\DeclareMathOperator{\GL}{GL}
\DeclareMathOperator{\Gal}{Gal}

\title{On the field generated by the periods of a Drinfeld module}
\author{Ernst-Ulrich Gekeler}
\date{\today}

\begin{document}

\begin{abstract}
	Generalizing the results of Maurischat in \cite{Maurischatxx}, we show that the field $K_{\infty}(\Lambda)$ of periods of a Drinfeld module $\phi$ of rank $r$ defined over $K_{\infty} = \mathds{F}_{q}((T^{-1}))$ may be arbitrarily large over $K_{\infty}$.
	We also show that, in contrast, the residue class degree $f( K_{\infty}(\Lambda) | K_{\infty})$ remains bounded by a constant that depends only on $r$.
\end{abstract}

\maketitle

\setcounter{section}{-1}

\section{Introduction}

The present paper is inspired by the recent note \cite{Maurischatxx} of Andreas Maurischat. Let $\phi$ be a Drinfeld $A$-module of rank $r \in \mathds{N}$ over $K_{\infty}$, where $A = \mathds{F}_{q}[T]$ and $K_{\infty} = \mathds{F}_{q}((T^{-1}))$. The field
$K_{\infty}(\Lambda)$ generated by the period lattice $\Lambda$ of $\phi$ agrees with its torsion field $K_{\infty}(\tor(\phi))$, and even with $K_{\infty}( \bigcup_{n \in \mathds{N}} \phi[T^{n}])$, where $\phi[T^{n}]$ is the module of $T^{n}$-torsion
of $\phi$.

In the case where $r = 2$ (the case $r=1$ is almost trivial), Maurischat shows through investigation of the Newton polygon of the $T$-division polynomial $\phi_{T}(X)$ of $\phi$ and its iterates that the ramification index 
$e( K_{\infty}(\Lambda) | K_{\infty})$ may become arbitrarily large. We generalize this fact to arbitrary ranks $r \geq 2$ (and to finite extensions $L$ of $K_{\infty})$, see Theorem \ref{Theorem.Ramification-over-L-unbounded}. We show moreover that the residue class degree
$f( L(\Lambda)|L)$ is bounded through a constant depending only on $r$ (Theorem \ref{Theorem.Residue-class-degree-bound}), thereby answering a question raised by several mathematicians, among which Bjorn Poonen, Chantal David, and Mihran Papikian.

Instead of the Newton polygon, we use the spectral filtration $\Lambda$ as our basic tool. Let $\mathbf{B} = \{ \lambda_{1}, \lambda_{2}, \dots, \lambda_{r}\}$ be a successive minimum basis of $\Lambda$ (see \cite{Gekeler2017} Section 3), grouped 
$\mathbf{B} = \mathbf{B}_{1} \cup \mathbf{B}_{2} \cup \dots \cup \mathbf{B}_{t}$ into subsets $\mathbf{B}_{t}$ of basis vectors of equal length, and let $\Lambda_{\tau}$ be the sublattice of $\Lambda$ generated by 
$\mathbf{B}_{1} \cup \dots \cup \mathbf{B}_{\tau}$, where $1 \leq \tau \leq t$. Put $L_{\tau} \defeq L(\Lambda_{\tau})$, so $L_{0} = L$, and $L_{t} = L(\Lambda)$. By means of an analysis of the intermediate extensions $L_{\tau}|L_{\tau-1}$ 
carried out in Section \ref{Section.Structure-of-the-field-extension}, we get control on the behavior of the ramification index and residue class degree along the tower $(L_{\tau})_{0 \leq \tau \leq t}$. Together with the determination of the spectrum 
$\spec_{A}(\Lambda) = (\lvert \lambda_{1} \rvert, \lvert \lambda_{2} \rvert, \dots, \lvert \lambda_{r} \rvert)$ of $\Lambda$ out of the Newton polygon $\mathrm{NP}(\phi_{\tau}(X))$ in a special case, performed in Section \ref{Section.An-example}, 
this yields our results Theorem \ref{Theorem.Ramification-over-L-unbounded} and Theorem \ref{Theorem.Residue-class-degree-bound}.

\begin{Notation} 
	\begin{itemize}
		\item $\mathds{F} = \mathds{F}_{q} = \text{finite field with $q$ elements, of characteristic $p$}$,
		\item $A = \mathds{F}[T]$, $K=\mathds{F}(T)$,
		\item $K_{\infty} = \mathds{F}((T^{-1})) = $ completion of $K$ w.r.t. the valuation $v=v_{\infty}$ of $K$ at infinity, with absolute value $\lvert \cdot \rvert$ normalized by $\lvert T \rvert = q$;
		\item $C = \widehat{\overline{K}}_{\infty} = $ completed algebraic closure of $K_{\infty}$, provided with the unique extension of $\lvert \cdot \rvert$ to $C$.
	\end{itemize}
	If $K_{\infty} \subset L \subset L'$ are finite extensions, $e(L'|L)$ and $f(L'|L)$ denote the ramification index and the residue class degree, respectively. Throughout, we fix one such $L$.
\end{Notation}

\section{Preliminaries}

\subsection{} The letter $\phi$ always denotes a Drinfeld $A$-module of rank $r \geq 2$ defined over $L$. It is described through its \textbf{$\boldsymbol{T}$-operator polynomial}.

\begin{equation}
	\phi_{T}(X) = TX + g_{1}X^{q} + \dots + g_{r}X^{q^{r}}, ~g_{1}, \dots, g_{r} \in L, ~g_{r} \neq 0
\end{equation}
or its \textbf{period lattice $\Lambda$}, a discrete free $A$-submodule of $C$ of rank $r$. These are related by 
\begin{align}
	e_{\Lambda}(Tz) &= \phi_{T}(e_{\Lambda}(z)) \qquad (z \in C),
\intertext{where}
	e_{\Lambda}(z) 	&= z {\prod_{\lambda \in \Lambda}}^{\prime} (1-z/\lambda) = z \prod_{0 \neq \lambda \in \Lambda} (1 - z/\lambda)
\end{align}
is the \textbf{exponential function} of $\Lambda$. See e.g. \cite{Goss96}, Chapter 4 for the elementary theory of Drinfeld modules and their uniformizations. Given $0 \neq a \in A$, we write $\phi[a]$ for the $A$-module of $a$-division points
of $\phi$, which is isomorphic with $(A/(a))^{r}$. It is easy to see that in fact $e_{\Lambda}$ as a power series in $z$ has coefficients in $L$, thus: \stepcounter{subsubsection} \stepcounter{subsubsection}
\begin{equation}
	\Lambda \text{ is contained in the separable closure $L^{\mathrm{sep}}$ of $L$}. 
\end{equation}
Hence $L(\Lambda)$ is a finite Galois extension of $L$. Our starting point is

\begin{Proposition} \label{Proposition.Description-of-L(Lambda)-as-L(tor(phi))}
	Let $a$ be any non-constant element of $A$. Then \stepcounter{equation}
	\[
		L(\Lambda) = \bigcup_{n \in \mathds{N}} L(\phi[a^{n}]) = L(\tor(\phi)).
	\]
	Here $\tor(\phi) = \bigcup_{0 \neq a \in A} \phi[a] \cong (K/A)^{r}$ is the torsion submodule of $\phi$, and $L(*)$ is the field extension of $L$ generated by $(*)$.
\end{Proposition}

\begin{proof}
	The first equality for $L = K_{\infty}$ and $a = T$ is shown in \cite{Maurischatxx} Proposition 2.1; the generalization to arbitrary $L$ and $a$ is obvious, as then is the second equality.
\end{proof}

In the sequel, we study the extension $L(\Lambda)|L$ that occurs in \ref{Proposition.Description-of-L(Lambda)-as-L(tor(phi))}, and its Galois group
\begin{equation}
	G = \Gal(L(\Lambda)|L).
\end{equation}

\subsection{} Let $\mathbf{B} = \{ \lambda_{1}, \lambda_{2}, \dots, \lambda_{r}\}$ be a \textbf{successive minimum basis} (SMB) of the lattice $\Lambda$. It is characterized by:
\subsubsection{} \label{Subsubsection.Characterisation-of-successive-minimum-basis} For each $i \in \{1,2,\dots,r\}$, $\lambda_{i} \in \Lambda$ has minimal absolute value $\lvert \lambda_{i} \rvert$ 
among the lattice vectors $\lambda \in \Lambda \smallsetminus \sum_{1 \leq j < i} A \lambda_{j}$. (The condition for $i=1$ means that $\lambda_{1}$ has minimal absolute value $> 0$.)

Choosing $\mathbf{B}$ according to \ref{Subsubsection.Characterisation-of-successive-minimum-basis}, it is in fact a basis of $\Lambda$, which is \textbf{orthogonal}, that is \stepcounter{subsubsection} \stepcounter{equation}
\begin{equation} \label{Equation.Orthogonality}
	\Big\lvert \sum_{1 \leq i \leq r} a_{i} \lambda_{i} \Big\rvert = \max_{i} \lvert a_{i} \rvert \lvert \lambda_{i} \rvert
\end{equation}
for each family $a_{1}, \dots, a_{r}$ of coefficients in $K_{\infty}$, and:
\subsubsection{} \label{Subsubsection.Series-of-absolute-values-of-lattice-vectors-independent-of-choices} The series $\lvert \lambda_{1} \rvert \leq \lvert \lambda_{2} \rvert \leq \dots \leq \lvert \lambda_{r} \rvert$ of values of the $\lambda_{i}$ is an invariant of $\Lambda$, that is, independent of choices made. (See \cite{Gekeler2017} Section 3 for a
proof of the assertions.) 

We call the $r$-tuple $(\lvert \lambda_{1} \rvert, \lvert \lambda_{2} \rvert, \dots, \lvert \lambda_{r} \rvert)$ the \textbf{spectrum $\spec_{A}(\Lambda)$} of the $A$-lattice $\Lambda$. For later use, we cite the following result,
which is discussed in \cite{Gekeler2018}, Observation 4.3. (Note the different numbering there, in which $\{ \lambda_{r}, \lambda_{r-1}, \dots, \lambda_{1}\}$ forms an SMB of $\Lambda$!)

\begin{Proposition} ~ \label{Proposition.Spectrum-and-newton-polygon}
	\begin{enumerate}[label=$\mathrm{(\roman*)}$]
		\item The spectrum $\spec_{A}(\Lambda)$ and the Newton polygon $\mathrm{NP}(\phi_{T}(X))$ of $\phi_{T}(X)$ determine each other.
		\item The point $(q^{i}, v(g_{i}))$, where $1 \leq i < r$, is a break point of $\mathrm{NP}(\phi_{T}(X))$ if and only if $\lvert \lambda_{i} \rvert < \lvert \lambda_{i+1} \rvert$. 
	\end{enumerate}
\end{Proposition}

The break points of $\mathrm{NP}(\phi_{T}(X))$, i.e., the sizes of the roots of $\phi_{T}(X)$, may be found from $\spec_{A}(\Lambda)$ by \cite{Gekeler2017} Lemma 3.4. However, the inverse map which describes 
$\spec_{A}(\Lambda)$ in terms of $\mathrm{NP}(\phi_{T}(X))$, is less explicit. See Section \ref{Section.An-example} for a comparatively simple special case.

\subsection{} \label{subsection.Collection-and-ordering-of-basis-vectors} We collect basis vectors of equal lengths in packets $\mathbf{B}_{\tau}$ of sizes $r_{\tau}$, and ordered according to size:
\[
	\underbrace{\lvert \lambda_{1} \rvert = \dots = \lvert \lambda_{r_{1}} \rvert}_{\mathbf{B}_{1}} < \underbrace{\lvert \lambda_{r_{1} + 1} \rvert = \dots = \lvert \lambda_{r_{1}+r_{2}} \rvert}_{\mathbf{B}_{2}} < \cdots \, \underbrace{= \lvert \lambda_{r} \rvert}_{\mathbf{B}_{t}}.
\]
That is $\mathbf{B} = \bigcupdot_{1 \leq \tau \leq t} \mathbf{B}_{\tau}$, and $r_{\tau}$ is the multiplicity by which the $\tau$-th of the values $\lvert \lambda_{i} \rvert$ counted without multiplicity occurs. We further let 
\begin{equation} \label{Equation.Direct-summand-of-Lambda}
	\overline{\mathbf{B}}_{\tau} = \mathbf{B}_{1} \cup \mathbf{B}_{2} \cup \dots \cup \mathbf{B}_{\tau} \quad \text{and} \quad \Lambda_{\tau} = \sum_{\lambda \in \overline{\mathbf{B}}_{\tau}} A\lambda \qquad (0 \leq \tau \leq t),
\end{equation}
a direct summand of $\Lambda$ with $\rk_{A}( \Lambda_{\tau}| \Lambda_{\tau-1}) = r_{\tau} (1 \leq \tau \leq t)$. The sequence of sublattices $\{0 \} = \Lambda_{0} \subset \Lambda_{1} \subset \dots \subset \Lambda_{t} = \Lambda$
is intrinsically defined and is called the \textbf{spectral filtration} on $\Lambda$. As the Galois group $G = \Gal(L(\Lambda) | L)$ acts $A$-linearly and length-preserving on $\Lambda$, it follows from \eqref{Equation.Orthogonality} that each $\sigma \in G$
satisfies
\begin{equation} \label{Equation.Galois-permutation-evaluated-on-basis-vector}
	\sigma(\lambda_{j}) = \sum_{1 \leq i \leq r} a_{i,j} \lambda_{i} \qquad (1 \leq j \leq r)
\end{equation}
with $a_{i,j} \in A$, where
\begin{equation} \label{Equation.Bound-for-matrix-entries}
	\lvert a_{i,j} \rvert \leq \lvert \lambda_{i} / \lambda_{j} \rvert.
\end{equation}
That is, $\sigma$ is represented by a $r\times r$-matrix of shape 
\begin{equation} \label{Equation.Block-matrix.structure}
	\begin{tikzpicture}[scale=0.5, baseline=(current  bounding  box.center)]
		\node at (2,2) (0) {\huge $0$};
		\node at (-3.5,3.5) (B1) {$B_{1}$};
		\node at (-2,2) (B2) {$B_{2}$};
		\node at (-2,-2) (aij) {$a_{i,j}$};
		\node at (3,-3) (Bt) {$B_{t}$};
		\node at (-5.5,-2) (i) {$i$};
		\node at (-2,-5.5) (j) {$j$}; 
		
		\draw (-5,-5) rectangle (5,5);
		\draw (-2.5,5) -- (-2.5,1.5) -- (-1.5,1.5) -- (-1.5,2.5) -- (-5,2.5) -- (-5,5) -- (-2.5,5) -- cycle;
		\draw (5,-5) -- (1,-5) -- (1,-1) -- (5,-1) -- (5,-5) -- cycle;
	
		\draw[fill=black, circle] (-1,1) circle (.75pt);
		\draw[fill=black, circle] (-0.25,0.25) circle (.75pt);
		\draw[fill=black, circle] (0.5,-0.5) circle (.75pt);	
		
		\draw[dotted] (-5,-2) -- (aij);
		\draw[dotted] (-2,-5) -- (aij);
	\end{tikzpicture}
\end{equation}
with matrices $B_{\tau} \in \GL(r_{\tau}, \mathds{F})$ along the diagonal, zeroes above the blocks $B_{\tau}$, and entries $a_{i,j} \in A$ bounded by \eqref{Equation.Bound-for-matrix-entries} below the blocks. We conclude:

\begin{Proposition}
	$G = \Gal(L(\Lambda)|L)$ is a finite group whose order is bounded by a constant that depends only on the spectrum $\spec_{A}(\Lambda)$ of $\Lambda$.
\end{Proposition}

Henceforth, we identify $G$ with the subgroup of $\GL(r,A)$ given by \eqref{Equation.Galois-permutation-evaluated-on-basis-vector}.

\begin{Example}[see {\cite{Maurischatxx}}, Theorem 3.1]
	Suppose that $r=2$, $\phi_{T}(X) = TX + gX^{q} + \Delta X^{q^{2}}$. There are two possibilities for the Newton polygon:
	\begin{enumerate}[label = (\alph*)]
		\item $\mathrm{NP}(\phi_{T}(X))$ is a straight line, that is, the point $(q, v(g))$ lies above or on the line determined by $(1,-1)$ and $(q^{2}, v(\Delta))$. Then, according to (\ref{Proposition.Spectrum-and-newton-polygon}), 
		$\spec_{A}(\Lambda) = (\lvert \lambda_{1} \rvert, \lvert \lambda_{1} \rvert)$ and $G \hookrightarrow \GL(2,\mathds{F})$.
		\item $(q, v(g))$ lies below the line through $(1,-1)$ and $(q^{2}, v(\Delta))$, and is a break point of $\mathrm{NP}(\phi_{T}(X))$. Then $\spec_{A}(\Lambda) = ( \lvert \lambda_{1} \rvert, \lvert \lambda_{2} \rvert)$ with
		$\lvert \lambda_{1} \rvert < \lvert \lambda_{2} \rvert$, and $G$ is a subgroup of 
		\[
			\left. \left\{ \begin{pmatrix} a & 0 \\ c & d \end{pmatrix} \, \right| \, a,d \in \mathds{F}^{*}, c \in A \text{ such that $\lvert c \rvert \leq \lvert \lambda_{2} / \lambda_{1} \rvert$} \right\}.
		\]
	\end{enumerate}
	In case (b) we still have to determine $\lvert \lambda_{1} \rvert$ and $\lvert \lambda_{2} \rvert$ from $\mathrm{NP}(\phi_{T}(X))$ (see Section \ref{Section.An-example}), and, in both cases, the actual size of $G$ inside the corresponding matrix group.
\end{Example}

\section{The structure of the field extension $L(\Lambda)|L$} \label{Section.Structure-of-the-field-extension}

We describe the tower of subfields of $L(\Lambda)|L$ that corresponds to the block structure \eqref{Equation.Block-matrix.structure} of $G$.

\subsection{} Having fixed some number $\tau$ with $1 \leq \tau \leq t$, we let $e_{\tau} = e_{\Lambda_{\tau}}$ be the exponential function associated with $\Lambda_{\tau}$ and $\phi^{(\tau)}$ the corresponding Drinfeld module, of rank \stepcounter{subsubsection}
\begin{equation} 
	\overline{r}_{\tau} = r_{1} + r_{2} + \dots + r_{\tau}.
\end{equation}
We also set $r_{0} = \overline{r}_{0} = 0$, $\Lambda_{0} = \{ 0\}$, $e_{0}(z) = z$, and $\phi^{(0)}$ is the trivial Drinfeld module of rank $0$, $\phi^{(0)}_{T}(X) = TX$. Further, \stepcounter{subsubsection}
\begin{equation}
	L_{\tau} \defeq L(\Lambda_{\tau})
\end{equation}
the field generated over $L$ by $\Lambda_{\tau}$, which by Galois theory equals the fixed field of the normal subgroup \stepcounter{subsubsection}
\begin{equation}
	G_{\tau}' \defeq \{ \sigma \in G \mid a_{i,j} = \delta_{i,j} \text{ if } 1 \leq i,j \leq \overline{r}_{\tau} \}
\end{equation}
of $G$. Here we have used the representation \eqref{Equation.Galois-permutation-evaluated-on-basis-vector} of $\sigma \in G$, and $\delta_{i,j}$ is the Kronecker delta. Then \stepcounter{subsubsection}
\begin{equation}
	G_{\tau} \defeq G/G_{\tau}' = \Gal(L_{\tau}|L)
\end{equation}
may be seen as a group of $\overline{r}_{\tau} \times \overline{r}_{\tau}$-matrices with a block structure similar to \eqref{Equation.Block-matrix.structure}. The canonical map
\[
	G_{\tau} = \Gal(L_{\tau}|L) \longrightarrow G_{\tau-1} = \Gal(L_{\tau-1} | L)
\]
corresponds to 
\[
	\begin{array}{|c|c|} \hline P & 0 \\ \hline Q & R \\ \hline \end{array} \longmapsto P,
\]
where $P$, $Q$, $R$ is an $\overline{r}_{\tau-1} \times \overline{r}_{\tau-1}$-matrix, $r_{\tau} \times \overline{r}_{\tau-1}$-matrix, $r_{\tau} \times r_{\tau}$-matrix, respectively. Let $U_{\tau} \subset G_{\tau}$ be the subgroup of matrices
with trivial $P$- and $R$-parts, i.e.,
\begin{equation}
	U_{\tau} = \{ \sigma \in G_{\tau} \mid a_{i,j} = \delta_{i,j} \text{ for } i \leq \overline{r}_{\tau-1} \text{ or } j > \overline{r}_{\tau-1} \}.
\end{equation}
In view of the given block structure, it is abelian and in fact an elementary abelian $p$-group. The kernel $\Gal(L_{\tau}| L_{\tau-1})$ contains $U_{\tau}$ as a normal subgroup. Define
\begin{equation}
	\tilde{L}_{\tau} \defeq L_{\tau}^{U_{\tau}} = \text{fixed field of $U_{\tau}$},
\end{equation}
and let
\begin{equation}
	H_{\tau} \text{ be the group $\Gal(L_{\tau}|L_{\tau-1})/U_{\tau} = \Gal(\tilde{L}_{\tau}|L_{\tau-1})$}. 
\end{equation}	
It equals the image of $G_{\tau}$ under the projection
\[
	\begin{array}{|c|c|} \hline P & 0 \\ \hline Q & R \\ \hline \end{array}   \longmapsto R,
\]
and is a subgroup of $\GL(r_{\tau}, A)$. Finally, define the subgroup
\begin{equation}
	\tilde{\Lambda}_{\tau} \defeq e_{\tau-1}(\Lambda_{\tau})
\end{equation}
of $L_{\tau}$.

\begin{Proposition} ~
	\begin{enumerate}[label=$\mathrm{(\roman*)}$] 
		\item $\tilde{\Lambda}_{\tau}$ is an $A$-submodule of $(L_{\tau}, \phi^{(\tau-1)})$, the additive group $L_{\tau}$ provided with the $A$-module structure via $\phi^{(\tau-1)}$.
		\item As such, $\tilde{\Lambda}_{\tau}$ is free of rank $r_{\tau}$.
		\item Let $\mathbf{B}_{\tau} = \{ \lambda_{i} \in \mathbf{B} \mid \overline{r}_{\tau-1} < i \leq \overline{r}_{\tau} \}$ be as in \ref{subsection.Collection-and-ordering-of-basis-vectors}. Then $\mathbf{B}_{\tau}$ maps under $e_{\tau-1}$ bijectively to an $A$-basis of $\tilde{\Lambda}_{\tau}$,
		and all the elements of $e_{\tau-1}(\mathbf{B}_{\tau})$ have the same absolute value in $L_{\tau}$, which is minimal among $\{ \lvert \lambda \rvert \mid 0 \neq \lambda \in \tilde{\Lambda}_{\tau} \}$.
	\end{enumerate}
\end{Proposition}

\begin{proof}
	$\tilde{\Lambda}_{\tau}$ is a well-defined subgroup of $L_{\tau}$, as $\Lambda_{\tau} \subset L_{\tau}$ and $e_{\tau-1}$ has coefficients in $L_{\tau-1} \subset L_{\tau}$. By definition, $e_{\tau-1}$ provides a morphism of $A$-modules
	from $L_{\tau-1}$ to $(L_{\tau-1}, \phi^{(\tau-1)})$, compatible with base extension to $L_{\tau}$, and
	\[
		\begin{tikzcd}
			e_{\tau-1} \colon \Lambda_{\tau}/\Lambda_{\tau-1} \ar[r, "\cong"] & \tilde{\Lambda}_{\tau},
		\end{tikzcd}
	\]	
	where the left hand side is isomorphic with $A^{r_{\tau}}$ by \eqref{Equation.Direct-summand-of-Lambda}. This shows (i) and (ii) and also that $e_{\tau-1}(\mathbf{B}_{\tau})$ is a basis. Let $\lambda \in \mathbf{B}_{\tau}$. From the orthogonality property \eqref{Equation.Orthogonality},
	\begin{equation} \label{Equation.Evaluation-Exponential-function}
		\lvert e_{\tau-1}(\lambda) \rvert = \lvert \lambda \rvert {\prod_{\mu \in \Lambda_{\tau}}}^{\prime} \lvert 1 - \lambda/ \mu \rvert = \lvert \lambda \rvert {\prod_{\mu \in \Lambda_{\tau},\, \lvert \mu \lvert \leq \lvert \lambda \rvert}}^{\prime} \lvert \lambda/\mu \rvert,
	\end{equation}
	which depends only on $\lvert \lambda \rvert$. Therefore, all the $\lvert e_{\tau-1}(\lambda) \rvert$ are equal and, in fact, minimal among $\{ \lvert e_{\tau-1}(\lambda) \rvert \mid \lambda \in \Lambda_{\tau} \smallsetminus \Lambda_{\tau-1}\}$.
\end{proof}	

\begin{Proposition} \label{Proposition.Field-equality}
	The fixed field $\tilde{L}_{\tau}$ of $U_{\tau}$ equals the field $L_{\tau-1}(\tilde{\Lambda}_{\tau})$ generated over $L_{\tau-1}$ by $\tilde{\Lambda}_{\tau}$.
\end{Proposition}

\begin{proof}
	By \eqref{Equation.Galois-permutation-evaluated-on-basis-vector} and the $\Lambda_{\tau-1}$-periodicity of $e_{\tau-1}$, $U_{\tau}$ fixes $\tilde{\Lambda}_{\tau}$. On the other hand, suppose that $\sigma \in G_{\tau}$ fixes $L_{\tau-1}(\tilde{\Lambda}_{\tau})$. Then 
	$\sigma(\lambda_{j}) = \lambda_{j}$ for $j \leq \overline{r}_{\tau-1}$ and 
	\[
		\sigma(\lambda_{j}) = \sum_{1 \leq i \leq \overline{r}_{\tau}} a_{i,j} \lambda_{i}
	\]
	for $\overline{r}_{\tau-1} < j \leq \overline{r}_{\tau}$ with $a_{i,j} = \delta_{i,j}$ if $\overline{r}_{\tau-1} < i \leq \overline{r}_{\tau}$, as $\sigma(\lambda_{j}) \equiv \lambda_{j} \pmod{\Lambda_{\tau-1}}$. Hence $\sigma \in U_{\tau}$.
\end{proof}

\subsection{} As $\tilde{\Lambda}_{\tau}$ is an $A$-sublattice in $(L_{\tau-1}, \phi^{(\tau-1)})$,
\begin{equation}
	\tilde{L}_{\tau} = L_{\tau-1}(\tilde{\Lambda}_{\tau-1}) = L_{\tau-1}(V_{\tau}),
\end{equation}
where $V_{\tau}$ is the $\mathds{F}$-vector space of dimension $r_{\tau}$ generated by $e_{\tau-1}(\mathbf{B}_{\tau})$. Further, $V_{\tau}$ is stable under $H_{\tau} = \Gal(\tilde{L}_{\tau}|L_{\tau-1})$, and thus
$H_{\tau}$ embeds into 
\[
	\GL(V_{\tau}) \overset{\cong}{\longrightarrow} \GL(r_{\tau}, \mathds{F}),
\] 
where the last isomorphism is via the choice of the ordered basis $e_{\tau-1}(\mathbf{B}_{\tau})$, i.e., the choice of $\mathbf{B}_{\tau}$.
This is in keeping with the fact the projection of $\sigma \in G$ onto its $\tau$-part (see \eqref{Equation.Block-matrix.structure}) is an element of $\GL(r_{\tau}, \mathds{F})$.

\begin{Definition}
	A finite $\mathds{F}$-vector subspace $V$ of $L^{\mathrm{sep}}$ is called \textbf{pure of weight $w(V) \in \mathds{Q}$} if all the non-vanishing $x \in V$ satisfy $\log x = \log_{q} \lvert x \rvert = w(V)$.
\end{Definition}

We summarize the preceding considerations as follows:

\begin{Proposition}
	For each $\tau$ with $1 \leq \tau \leq t$ the intermediate field $\tilde{L}_{\tau}$ of $L_{\tau}|L_{\tau-1}$ satisfies:
	\begin{enumerate}[label=$\mathrm{(\roman*)}$]
		\item $\tilde{L}_{\tau} = L_{\tau-1}(\tilde{\Lambda}_{\tau}) = L_{\tau-1}(V_{\tau})$ is obtained by adjoining the $H_{\tau}$-stable pure $\mathds{F}$-subspace $V_{\tau}$ of dimension $r_{\tau}$ (and a certain weight given by \eqref{Equation.Evaluation-Exponential-function}).
		\item The Galois group $U_{\tau}$ of $L_{\tau}$ over $\tilde{L}_{\tau}$ is an elementary abelian $p$-group.
	\end{enumerate}
\end{Proposition}

\begin{Corollary} \label{Corollary.Ramification-index}
	With notations as before, the ramification index $e(L_{\tau}|\tilde{L}_{\tau})$ equals $\# U_{\tau}$ or $p^{-1} \# U_{\tau}$.
\end{Corollary}

\begin{proof}
	Consider the filtration of $U_{\tau}$ given by the higher ramification groups $U_{\tau, i}$ (see \cite{Serre68} Chapitre 4), $i = -1,0,1,\dots$. Then
	\begin{itemize}
		\item[] $U_{\tau,-1} = U_{\tau}$; $U_{\tau}/U_{\tau,0}$ is cyclic;
		\item[] $U_{\tau,0}/U_{\tau,1}$ is cyclic of order prime to $p$, hence trivial, and 
		\item[] all the higher $U_{\tau,i}/U_{\tau,i+1}$ are $p$-groups.
	\end{itemize}
	Hence, as $U_{\tau}$ is $p$-elementary abelian, \textbf{either} $U_{\tau} = U_{\tau,0}$ and $L_{\tau}|\tilde{L}_{\tau}$ is completely ramified of degree $\# U_{\tau}$ \textbf{or} the residue class degree
	$[U_{\tau} : U_{\tau,0}]$ is $p$ and the ramification index equals $p^{-1} \# U_{\tau}$.
\end{proof}

\subsection{} Let's have a closer look at the extension $\tilde{L}_{\tau} = L_{\tau-1}(V_{\tau})$ of $L_{\tau-1}$. We assume that all the field extensions considered are subfields of the separable closure $L^{\mathrm{sep}}$ of $L$. This
identifies the residue class fields with subfields of the residue class field of $L^{\mathrm{sep}}$, which is a separable closure $\mathds{F}^{\mathrm{sep}}$ of $\mathds{F}$.

Let $d \in \mathds{N}$ be the precise denominator of the weight $w(V_{\tau})$, and let $L_{\tau-1}'$ be a completely ramified separable extension of $L_{\tau-1}$ of degree $d$. Then there exists $y \in L_{\tau-1}'$ with
$v(y) = w(V_{\tau}) = v(x)$ for each $0 \neq x \in V_{\tau}$. Replacing $V_{\tau}$ with $V_{\tau}' \defeq y^{-1} V_{\tau}$,
\begin{equation}
	L_{\tau-1}'(V_{\tau}) = L_{\tau-1}'(V_{\tau}'), \quad \text{where } w(V_{\tau}') = 0.
\end{equation}
Now, since $v(x-x') = v(x) = v(x')$ for $x \neq x' \in V_{\tau} \smallsetminus \{0\}$, different elements of $V_{\tau}'$ reduce to different elements in $\mathds{F}^{\mathrm{sep}}$, that is, reduction modulo the maximal ideal yields
an $\mathds{F}$-isomorphism
\begin{equation}
	\begin{tikzcd}
		V_{\tau}' \ar[r, "\cong"]	& \overline{V}_{\tau}' \ar[r, hook] &\mathds{F}^{\mathrm{sep}}.
	\end{tikzcd}
\end{equation}

Let $\mathds{F}_{\tau-1}$, $\mathds{F}_{\tau-1}'$ be the residue class fields of $L_{\tau-1}$, $L_{\tau-1}'$, respectively. Now:
\begin{itemize}
	\item[] $V_{\tau}$ is stable under $H_{\tau} = \Gal(\tilde{L}_{\tau}|L_{\tau-1})$, so
	\item[] $V_{\tau}'$ is stable under $H'_{\tau} \defeq \Gal(L_{\tau-1}'(V_{\tau}') | L_{\tau-1}')$, and
	\item[] $\overline{V}_{\tau}'$ is stable under $\Gal( \mathds{F}_{\tau-1}'(\overline{V}_{\tau}')| \mathds{F}_{\tau-1}')$.
\end{itemize}
The latter is generated by a suitable Frobenius element (i.e., $x \mapsto x^{\#(\mathds{F}_{\tau-1}')}$), which by our choice of the basis $\mathbf{B}_{\tau}$ for $V_{\tau}$ (thus bases for $V_{\tau}'$ and $\overline{V}_{\tau}'$)
corresponds to a matrix in $\GL(r_{\tau}, \mathds{F})$. Let $f_{\tau}$ be its multiplicative order. Then
\begin{equation}
	f_{\tau} = [\mathds{F}_{\tau-1}'(\overline{V}_{\tau}') : \mathds{F}_{\tau-1}'] = [L_{\tau-1}'(V_{\tau}') : L_{\tau-1}'],
\end{equation}
where the last equality comes from Hensel's Lemma. In particular, $L_{\tau-1}'(V_{\tau}') = L_{\tau-1}'(V_{\tau})$ is unramified of degree $f_{\tau}$ over $L_{\tau-1}'$. Consider the diagram
\begin{equation}
	\begin{tikzcd}[row sep=small]
																										& L_{\tau-1}'(V_{\tau}) \ar[dd, dash] 	\\
		\tilde{L}_{\tau} = L_{\tau-1}(V_{\tau}) \ar[ur, dash]	\ar[dd, dash]	&														\\
																										& L_{\tau-1}'										\\
		M_{\tau} \ar[ur, dash]	\ar[d, dash]											& \\
		L_{\tau-1} \ar[uur, dash]
	\end{tikzcd}	
\end{equation}
where $M_{\tau} \defeq L_{\tau-1}(V_{\tau}) \cap L_{\tau-1}'$. We see from the above:

\begin{Proposition} \label{Proposition.Isomorphies-of-special-Galois-groups}
	\[
		\begin{tikzcd}
			\Gal(\mathds{F}_{\tau-1}'(\overline{V}_{\tau-1}')|\mathds{F}_{\tau-1}') \ar[r, "\cong"]	& \Gal(L_{\tau-1}'(V_{\tau}) | L_{\tau-1}') \ar[r, "\cong"] & \Gal(L_{\tau-1}(V_{\tau})|M_{\tau})),
		\end{tikzcd}
	\]
	where the groups are cyclic of order $f_{\tau}$, $L_{\tau-1}(V_{\tau})|M_{\tau}$ is unramified and $M_{\tau}|L_{\tau-1}$ is completely ramified.
\end{Proposition}	

\begin{Remark}
	Note that neither $L_{\tau-1}'$ nor $M_{\tau}$ is canonically determined; only their degrees over $L_{\tau-1}$ are intrinsically characterized by $[L_{\tau-1}' : L_{\tau-1}] = d_{\tau} = \mathrm{denom}(w(V_{\tau}))$ and
	$[M_{\tau}: L_{\tau-1}] = e(\tilde{L}_{\tau}| L_{\tau-1})$.
\end{Remark}

\section{An example} \label{Section.An-example}

\subsection{} We consider in more detail the special case where the spectrum $\spec_{A}(\Lambda)$ of $\Lambda$ is $(\lvert \lambda_{1} \rvert < \lvert \lambda_{2} \rvert = \dots = \lvert \lambda_{r} \rvert$). By Proposition \ref{Proposition.Spectrum-and-newton-polygon}, this corresponds
to the following behavior of the Newton polygon:
\begin{equation} \stepcounter{subsubsection} \stepcounter{subsubsection}
	 \text{The break points of $\mathrm{NP}( \Phi_{T}(X))$ are } 	(1,-1), ~(q, v(g_{1})), ~(q^{r}, v(g_{r})).
\end{equation}
Such behavior is realized in particular by the Drinfeld module $\phi$ with
\begin{equation} \label{Equation.Phi-equation-drinfeld-module}
	\phi_{T}(X) = TX + gX^{q} + \Delta X^{q^{r}} \qquad (\Delta \neq 0),
\end{equation}
where the intermediate terms $g_{i}X^{q^{i}}$ ($1 < i < r$) vanish and $(q, v(g))$ lies below the line joining $(1,-1)$ and $(q^{r}, v(\Delta))$. As $\spec_{A}(\Lambda)$ and the properties derived from it depend only on $\mathrm{NP}(\phi_{T}(X))$,
we assume for the rest of this section that $\phi$ is given by \eqref{Equation.Phi-equation-drinfeld-module}. Our aim is to determine $\lvert \lambda_{1} \rvert$ and $\lvert \lambda_{2} \rvert$ from $v(g)$ and $v(\Delta)$, and to draw conclusions. Also, for psychological reasons, we mainly work
with $\log x = \log_{q} \lvert x \rvert = -v(x)$ instead of the valuation $v$. 

\begin{Remark} \label{Remark.Drinfeld-modules-1-spars}
	Rank-$r$ Drinfeld modules given by \eqref{Equation.Phi-equation-drinfeld-module} are called \textbf{$\boldsymbol{1}$-sparse} in \cite{Gekeler2017}.
\end{Remark}

\begin{Proposition} \label{Proposition.Log-forumla}
	In the given situation, put $j \defeq g^{(q^{r}-1)/(q-1)}/\Delta$. Let $\{ \lambda_{1}, \lambda_{2}, \dots, \lambda_{r}\}$ be an SMB of the period lattice $\Lambda$ of $\phi$, 
	$\lvert \lambda_{1} \rvert < \lvert \lambda_{2} \rvert = \dots = \lvert \lambda_{r} \rvert$. Define $s \defeq \log(\lambda_{2}/\lambda_{1})$, with integral part $n = [s]$. Then the formula
	\begin{equation} \label{Proposition.Eq.log-evaluation}
		\log j = (q^{r-1}-1)\left(\frac{1}{q-1} + s-n\right)q^{n+1}
	\end{equation}
	holds.
\end{Proposition}	

\begin{proof}
	This is formula (i) in Corollary 5.7 of \cite{Gekeler2017}. (Note the different numbering in \cite{Gekeler2017}, where $\lambda_{r}, \lambda_{r-1}, \dots, \lambda_{1}$ in this order is an SMB.)
\end{proof}

We will use \ref{Proposition.Log-forumla} to find lower bounds for the ramification index of $L(\lambda_{1}, \lambda_{2})$ over $L(\lambda_{1})$, and thus for the degree of $L(\Lambda)$ over $L$.
\subsection{} For $1 \leq i \leq r$ let
\begin{equation} 
	\mu_{i} = e_{\Lambda}(\lambda_{i}/T).
\end{equation}
Then $\{ \mu_{1}, \dots, \mu_{r}\}$ is an $\mathds{F}$-basis of $\phi[T]$ and 
\begin{equation}
	\lvert \mu_{1} \rvert < \lvert \mu_{2} \rvert = \dots = \lvert \mu_{r} \rvert
\end{equation}
holds, as well as
\begin{equation} \label{Eq.Evalution-sum-mui}
	\Big\lvert \sum_{1 \leq i \leq r} a_{i}\mu_{i} \Big\rvert = \begin{cases} \lvert \mu_{1} \rvert,	&a_{2} = \dots = a_{r} = 0, \\ \lvert \mu_{2} \rvert, 	&\text{otherwise} \end{cases}
\end{equation}
for coefficients $a_{i} \in \mathds{F}$, not all zero, see \cite{Gekeler2017} Lemma 3.4. One easly derives from \eqref{Eq.Evalution-sum-mui} and \eqref{Equation.Orthogonality} that 
\[
	\lvert \mu_{1} \rvert = \lvert \lambda_{1}/T \rvert, \text{ i.e., } \log \mu_{1} = \log \lambda_{1} - 1
\]
and
\[
	\lvert g/T \rvert = \lvert \mu_{1} \rvert^{1-q}, \text{ i.e., } (q-1) \log \mu_{1} = 1 - \log g,
\]
that is
\begin{equation} \label{Eq.Equality-lambda1}
	\log \lambda_{1} = (q - \log g)/(q-1).
\end{equation}
Hence \eqref{Proposition.Eq.log-evaluation} together with \eqref{Eq.Equality-lambda1} yields $\log g$ and $\log \Delta$ in terms of $\lvert \lambda_{1} \rvert$ and $\lvert \lambda_{2} \rvert$.
\subsection{} Let $\varphi$ be the real function
\begin{equation}
	s \longmapsto (q^{r-1} - 1) \left( \frac{1}{q-1} + s- n\right)q^{n+1}
\end{equation}
that occurs in \eqref{Proposition.Eq.log-evaluation}, i.e., $\log j = \varphi(s)$. It is continuous, monotonically increasing, convex, piecewise linear, and bijective from the real interval $[0,\infty)$ to $[\varphi(0), \infty)$, with $\varphi(0) = q(q^{r-1}-1)/(q-1)$. Let
\begin{equation}
	\begin{tikzcd}
		{\psi \colon [\varphi(0), \infty)} \ar[r, "{\cong}"] & {[0,\infty)}
	\end{tikzcd}
\end{equation}
be its inverse; it has similar properties, with \enquote{convex} replaced with \enquote{concave}. Then $s = \psi(\log j)$, i.e.,
\begin{equation}
	\log \lambda_{2} = \log \lambda_{1} + \psi(\log j) = (q- \log g)/(q-1) + \psi( \log j).
\end{equation}
In particular, as for fixed $g$, $\log j$ can be made arbitrarily large letting $\Delta$ tend to zero, also $\log \lambda_{2}$ can be made arbitrarily large.

Assuming $\log g = 0$ and $\log \Delta$ very small with numerator coprime to $p$, the formula derived from \eqref{Proposition.Eq.log-evaluation}
\begin{equation}
	-\log \Delta / q^{n+1} = (q^{r-1}-1)/(q-1) + (q^{r-1} - 1)(s-n)
\end{equation}
shows that $q^{n+1}$ (with $n \gg 0$) divides the denominator of $s = \log( \lambda_{2}/\lambda_{1})$. Hence the ramification index of $L(\lambda_{1}, \lambda_{2})$ over $K_{\infty}$, and in particular the extension degree
$[L(\lambda_{1}, \lambda_{2}) : K_{\infty}]$ is divisible by $q^{n+1}$. Thus we have proved:

\begin{Theorem} \label{Theorem.Ramification-over-L-unbounded}
	The field $L(\Lambda) = L(\tor(\phi))$ of periods of $\phi$, where $\phi$ is a Drinfeld $A$-module of rank $r \geq 2$ over $L$, may have arbitrarily large ramification index over $L$. In particular, its degree $[L(\Lambda) : L]$ is unbounded.
\end{Theorem}

\section{The residue class degree}

In contrast with the ramification index, the residue class degree $f(L(\Lambda)|L)$ is bounded by a constant that depends only on the rank $r$ of the Drinfeld module.

\begin{Theorem} \label{Theorem.Residue-class-degree-bound}
	Let $\phi$ be a rank-$r$ Drinfeld $A$-module over $L$ with period lattice $\Lambda$, provided with a successive minimum basis $\mathbf{B} = \{ \lambda_{1}, \lambda_{2}, \dots, \lambda_{r} \}$. Split $\mathbf{B} = \mathbf{B}_{1} \cup \dots \cup \mathbf{B}_{t}$ as in \ref{subsection.Collection-and-ordering-of-basis-vectors}, where $\# \mathbf{B}_{\tau} = r_{\tau}$ and $\sum_{1 \leq \tau \leq t} r_{\tau} = r$. The residue class degree $f(L(\Lambda)|L)$ is bounded by 
	\begin{equation}
		f(L(\Lambda)|L) \leq \prod_{1 \leq \tau \leq t} (q^{r_{\tau}} - 1) \cdot p^{t-1}.
	\end{equation}
\end{Theorem}

\begin{proof}
	Consider the tower of subfields described in Section \ref{Section.Structure-of-the-field-extension}:
	\[
		L = L_{0} \subset L_{1} = L(\Lambda_{1}) \subset \dots \subset L_{\tau} = L(\Lambda_{\tau}) \subset \dots \subset L_{t} = L.
	\]
	Each step $L_{\tau}|L_{\tau-1}$ is subdivided 
	\[
		L_{\tau-1} \subset M_{\tau} \subset \tilde{L}_{\tau} \subset L_{\tau} \qquad (\tau = 1,2, \dots, t).
	\]
	Now:
	\begin{itemize}
		\item $M_{\tau}|L_{\tau-1}$ is completely ramified (Proposition \ref{Proposition.Isomorphies-of-special-Galois-groups});
		\item $\tilde{L}_{\tau}|M_{\tau}$ is unramified of degree $f_{\tau}$, where $f_{\tau} = \mathrm{ord}(x)$ with some $x \in \GL(r_{\tau}, \mathds{F})$ (Proposition \ref{Proposition.Isomorphies-of-special-Galois-groups});
		\item $L_{\tau}|\tilde{L}_{\tau}$ has residue class degree 1 or $p$ (Corollary \ref{Corollary.Ramification-index}) and $\tilde{L}_{1} = L_{1}$, since $\tilde{\Lambda}_{1} = \Lambda_{1}$.
	\end{itemize}
	The result now follows by the multiplicativity of the residue class degree in towers and from the elementary fact that
	\begin{equation}
		\mathrm{ord}(x) \leq q^{n} - 1 \qquad \text{for } x \in \GL(n, \mathds{F}). \qedhere
	\end{equation}
\end{proof}

\begin{Remarks}
	\begin{enumerate}[label=(\roman*), wide]
		\item As $\sum r_{\tau} = r$, $\prod (q^{r_{\tau}} - 1) \leq q^{r} - 1$, so $f(L(\Lambda)|L)$ is always less or equal to $(q^{r} - 1)p^{t-1}$.
		\item As $L(\Lambda) = L(\tor(\phi))$, the constant field extension in the torsion field $L(\tor(\phi))$, i.e., the algebraic closure $\overline{\mathds{F}}$ of $\mathds{F}$ in $L(\tor(\phi))$, is finite. Its degree
		$[\overline{\mathds{F}} : \mathds{F}]$ is bounded by $f(L|K_{\infty})(q^{r} - 1)p^{t-1}$. In particular, if $\phi$ is defined over a finite extension $K'$ of $K$, and $\overline{\mathds{F}} \defeq$ algebraic closure of $\mathds{F}$
		in $K'(\tor(\phi))$, then 
		\begin{equation}
			[\overline{\mathds{F}} : \mathds{F}] \leq \min_{w \mid \infty} f(K_{w}' : K_{\infty}))(q^{r} - 1) p^{r-1},
		\end{equation}
		where $w$ runs through the places of $K'$ above the place $\infty$ of $K$.
	\end{enumerate}
\end{Remarks}

\end{document}